\documentclass[
final
 , nomarks
]{dmtcs-episciences}

\usepackage[utf8]{inputenc}
\usepackage{subfigure}
\usepackage{url}
\usepackage{latexsym}
\usepackage{amsfonts,amsmath,amssymb,amsthm}

\usepackage[round]{natbib}

\newcommand{\av}{\operatorname{Av}}
\newtheorem{theorem}{Theorem}[section]
\newtheorem{proposition}[theorem]{Proposition}
\newtheorem{lemma}[theorem]{Lemma}
\newtheorem{corollary}[theorem]{Corollary}
\newtheorem{question}[theorem]{Question}

\newtheorem{conjecture}[theorem]{Conjecture}

{

}

{

\author[Mikl\'os B\'onal]{Mikl\'os B\'ona\affiliationmark{1}\thanks{Supported by Simons Collaboration Award 940024.}
  }
\title[Long Increasing Subsequences]{Long Increasing Subsequences and Non-algebraicity}
\affiliation{
  University of Florida, Gainesville, Florida, USA\\}
\begin{document}
\publicationdata{vol. 26:1, Permutation Patterns 2023}{2024}{3}{10.46298/dmtcs.12539}{2023-11-11; 2023-11-11; 2024-02-19; 2024-06-18}{2024-06-18}
\maketitle
\begin{abstract}
 We use a recent result of Alin Bostan to prove that the generating functions of two infinite sequences of permutation classes are not algebraic. 
\end{abstract}

\section{Introduction}

 We say that a permutation $p$ {\em contains} the pattern (or subsequence) $q=q_1q_2\cdots q_k$ 
if there is a $k$-element set of indices $i_1<i_2< \cdots <i_k$ such that $p_{i_r} < p_{i_s} $ if and only
if $q_r<q_s$.  If $p$ does not contain $q$, then we say that $p$ {\em avoids} $q$. For example, $p=3752416$ contains
$q=2413$, as the first, second, fourth, and seventh entries of $p$ form the subsequence 3726, which is order-isomorphic
to $q=2413$.  A recent survey on permutation 
patterns by Vatter can be found in \cite{v15}. Let  $\av_n(q)$ be 
the set of permutations of {\em length} $n$ that avoid the pattern $q$, where the {\em length} of a permutation is the number of entries in it.
 If $S$ is a {\em set} of patterns, and the permutation $p$ avoids all patterns in $S$, then we will say that $p$ \emph{avoids} $S$, and 
we will write $|\av_n(S)|$ for the number of such permutations of length $n$, $\av(S)$ for the set of such permutations of all lengths (such a set is called a \emph{permutation class}) and
 $\av_n(S)$ for those such permutations of length $n$. 

 In general, it is very difficult to compute
the numbers  $|\av_n(S)|$, or to describe their sequence as $n$ goes to infinity. Recently, there has been some progress in proving {\em negative results} about the ordinary generating
function $A_S(z)$ of the sequence $|\av_n(S)|$.  In \cite{b20}, the present author proved that for most patterns $q$, the generating function 
$A_q(z)=\sum_{n\geq 0} |\av_n(q)|z^n$ is not rational. Non-algebraicity of these generating functions is also very hard to prove, because there are very few general tools to prove that a combinatorial power 
series is not algebraic. As we explain in the next section, most results on non-algebraicity of generating functions $A_S(z)=\sum_{n\geq 0} |\av_n(S)|z^n$
were based on exact asymptotics of the coefficients, and those exact asymptotics are very hard to establish. (There has been one example \cite{gp} when non-algebraicity of 
the generating function of a permutation class was shown as a consequence of the stronger result that the generating function was not differentiably finite.)
In this paper, we will use a recent result of Alin Bostan to prove the non-algebraicity of $A_S(z)$ for two infinite sequences of sets $S$ from a weaker condition on the growth rate
of the numbers $|\av_n(S)|$.

\section{Tools to Prove Non-algebraicity} 
A power series $A(z)$ is called {\em algebraic} if there are polynomials $P_0(z), P_1(z),\cdots, P_d(z)$ that are not all identically zero so that the equality 
\[P_0(z)+P_1(z)A(z)+P_2(z)A^2(z)+\cdots +P_d(z)A^d(z)=0\] holds. See Section 6 of \cite{s23} for a high-level introduction to the theory of algebraic power series. 
Until recently, the only general, direct method to prove non-algebraicity of a generating function $A_q(S)$ was the following theorem of  \cite{j31}. 

\begin{theorem} \label{jungen}  Let $m$ be a positive integer, let $c$ and $\gamma$ be positive constants,    and let $A(z)=\sum_{n\geq 0}a_nz^n$ be a power series with complex coefficients. 
If \[a_n \simeq c \frac{ \gamma^n}{n^m} ,\]
then $A(z)$ is not an algebraic power series. \end{theorem}

The following theorem of Amitaj Regev makes Theorem \ref{jungen} immediately applicable for our purposes. 
\begin{theorem}[\cite{r81}] \label{regev}
For all $k\geq 2$, there exists a constant $r_k$ so that the asymptotic equality
\[\mid \av_n(12\cdots k) \mid \simeq r_k \frac{(k-1)^{2n}}{n^{(k^2-2k)/2} }\]
holds.
\end{theorem}  

\begin{corollary} \label{evenmonotone} Let $k>2$ be an even integer, and let $q=12\cdots k$. Then $A_q(z)$ is not algebraic. \end{corollary}
\begin{proof} If $k>2$ is even, then $(k^2-2k)/2$ is a positive integer, and so Theorem \ref{jungen} implies that $A_q(z)$ is not algebraic.
\end{proof}

It is usually difficult to find exact asymptotics for the numbers $|\av_n(S)|$, and therefore direct applications of Theorem \ref{jungen} to prove non-algebraicity of $A_q(S)$ for other 
pattern classes are rare.

The following result is due to Alin Bostan \cite{bos21}. See Lemma 6.3 in \cite{bb22} for its proof.  Compared to Theorem \ref{jungen},  it relaxes the asymptotics criterion on the coefficients of $A(z)$ somewhat, while yielding the
same conclusion.

\begin{lemma} \label{bostan}
 Let $A(z)=\sum_{n\geq 0}a_nz^n$ be a power series with nonnegative real coefficients that is analytic at the origin. Let us assume that constants $c$, $C$, $K$ and $m$ exist so that $m>1$ is an integer, and for all positive integers $n$, the chain of inequalities
\begin{equation} \label{condition}  
c \frac{K^n}{n^m}  \leq a_n  \leq C \frac{K^n}{n^m}
\end{equation} 
holds. Then $A(z)$ is not an algebraic power series.
\end{lemma}

In other words, in order to prove non-algebraicity of $A(z)$, we do not have to prove that the numbers $a_n$ are asymptotically equal to $\gamma \cdot K^n/n^m$; it suffices to show
that they are {\em between two constant multiples} of $K^n/n^m$. This is what we will do in the next section.

\section{Patterns with long increasing subseqences}
\subsection{The case of length five}
In \cite{bp23}, Jay Pantone and the present author studied classes of permutations avoiding patterns with long increasing subsequences. 
In particular, they considered the set of patterns $A_{k,k}$ consisting of the $k-1$ patterns of length $k$ that start with an increasing subsequence of length $k-1$ and end
in an entry less than $k$. For instance, $A_{5,5}=\{12354,12453,13452,23451\}$.  They proved that for $k\geq 3$, the exponential order of the sequence
$|\av_n(A_{k,k})|$ is $(k-2)^2+1$, that is, $\lim_{n\rightarrow \infty} \sqrt[n]{|\av_n(A_{k,k})|} = (k-2)^2+1$. Based on numerical evidence, they made the following conjecture.

\begin{conjecture} \label{asym} There exists a constant $R$ so that
\[\mid \av_n(A_{5,5}) \mid  \simeq R \cdot \frac{10^n}{n^4} .\]
\end{conjecture}

Theorem \ref{jungen} shows that  Conjecture \ref{asym} directly implies the following conjecture.

\begin{conjecture} \label{nonalg} 
The generating function $A_{A_{5,5}}(z)$ is not algebraic.
\end{conjecture} 

In this paper, we are going to prove Conjecture \ref{nonalg} without first proving Conjecture \ref{asym}. We will then prove an analogous result for a second infinite sequence of
patterns.

 Let $p=p_1p_2\cdots p_n$ be a permutation.
For any entry $p_h$ of $p$, let the {\em rank} of $p_h$ be the length of the longest increasing subsequence of $p$ that ends in $p_h$. 
Now let us assume that $p$ avoids $A_{5,5}$. Let $p_j$ be the leftmost entry of $p$ that is of rank 4. Then $p_{j+1}>p_j$, or a forbidden pattern is formed.
Similarly, $p_{j+2}>p_{j+1}$ or a forbidden pattern is formed, and so on. So the subsequence $p_jp_{j+1}\cdots p_n$ is an increasing subsequence. 
In other words, each permutation $p\in \av_n(A_{5,5})$ naturally decomposes into two parts; the 1234-avoiding permutation $p_1p_2\cdots p_{j-1}$, which one might call
the {\em front} and the
increasing subsequence $p_jp_{j+1}\cdots p_n$ that one might call the {\em tail}. Note that $j\geq 4$, and if $p$ avoids 1234, then the tail is empty. (In this case, we can set $j=n+1$.)

This leads to the following lemma. Let $f_n=|\av_n(1234)|$ for shortness.
\begin{lemma} \label{firstchain}
The chain of inequalities
\[    \frac{1}{4}  \sum_{i=3}^n  f_{i} {n\choose n-i}  \leq \mid  \av_n(A_{5,5}) \mid \leq  \sum_{i=3}^n  f_{i} {n\choose n- i}  \] holds.
\end{lemma}

\begin{proof} Let $i$ be a integer so that $3\leq i\leq n$. 
Choosing an $n-i$-element subset $T$ of the set $[n]=\{1,2,\cdots ,n\}$, constructing a 1234-avoiding permutation on $[n]\setminus T$ (the ``front part'')
 and postpending it by the elements of $T$ 
written in increasing order (the``back part''), we get a permutation in $\av_n (A_{5,5})$. We get every permutation in  $\av_n (A_{5,5})$ at most four times in this way, because 
for a permutation  $p\in \av_n (A_{5,5})$, the front part of $p$ cannot contain an increasing subsequence of length four or more. In other words, if the last $a$ entries of $p$ form 
an increasing subsequence, but the last $a+1$ entries do not, then there are at most four choices for the number $n-i$ above, namely $a$, $a-1$, $a-2$, and $a-3$. 
of legth $m$ On the other hand, we get every permutation 
 $p\in \av_n (A_{5,5})$ at least once in this way. Indeed, if  $p=p_1p_2\cdots p_n\in \av_n(A_{5,5})$, and $p_j$ is the leftmost entry of rank 4, then $p$ is obtained by
setting $T$ to be the underlying set of the entries $p_jp_{j+1}\cdots p_n$, and selecting the 1234-avoiding permutation $p_1p_2\cdots p_{j-1}$ on the set $[n]\setminus T$.
If $p$ avoids 1234, then $j$ is undefined; in this case we choose $T$ to be the empty  set, and we choose $p$ itself on the set $[n]\setminus T$.
\end{proof}

Lemma \ref{firstchain} squeezes the number $|\av_n(A_{5,5})|$ between two constant multiples of the sum $s_n=\sum_{i=3}^n {n\choose n- i} f_i  =\sum_{i=3}^n {n\choose i} f_i$.
All we need in order to be able to use Lemma \ref{bostan} is to prove that $s_n$ is between two constant multiples of $  \frac{10^n}{n^4}$. We will prove this in an elementary way, 
in two simple propositions, but after those propositions, we point the interested reader to the direction of a more high-brow approach.
Note that Theorem \ref{regev} directly implies that there are absolute constants $\alpha$ and $\beta$ so that
 \begin{equation} \label{defalpha} \alpha \cdot \frac{9^i}{i^4} \leq  f_i \leq \beta \cdot \frac{9^i}{i^4} \end{equation} for all $i$.

 Recall that $f_n=|\av_n(1234)|$ and that $s_n=f_i \sum_{i=3}^n {n\choose n-i} =\sum_{i=3}^{n} f_i{n\choose i}$. 

\begin{proposition} \label{first}
There exists an absolute constant $C_1>0$ so that $s_n\leq C_1 \cdot \frac{10^n}{n^4}$.
\end{proposition}

\begin{proof} Recall  that   $\beta$ was defined in the previous paragraph, and in (\ref{defalpha}). 
Let us split up $s_n$ to two parts, based on whether $i\leq  n/2$. If $i\leq  n/2$, then we have
\begin{eqnarray*}   \sum_{i=3}^{n/2} {n\choose i} f_i & \leq &  \beta \cdot \sum_{i=3}^{n/2}  \frac{9^i}{i^4} {n\choose i} \\
& \leq & \beta \cdot  \frac{n}{2}  9^{n/2} {n\choose n/2} \\
& \leq & \beta \cdot \frac{n}{2} \cdot  3^n \cdot 2^n  \\
& \leq & C_2 \cdot \frac{10^n}{n^4}.
\end{eqnarray*}

For the part of the sum $s_n$ where $i\geq n/2$, we have
\begin{eqnarray*}   \sum_{n/2}^{n} {n\choose i} f_i & \leq &  \beta \cdot \sum_{n/2}^{n}  \frac{9^i}{i^4} {n\choose i} \\
& \leq & \beta  \sum_{i=n/2}^{n}  \frac{9^i}{(n/2)^4} {n\choose i} \\
&\leq & \frac{C_3}{n^4}  \sum_{i=n/2}^{n} 9^i {n\choose n-i} \\
&\leq & C_3 \frac{10^n}{n^4}.
\end{eqnarray*}
Setting $C_1=C_2+C_3$ completes the proof. 
\end{proof}

\begin{proposition} \label{second}
There exists an absolute constant $c_1>0$ so that $s_n\geq c_1 \cdot \frac{10^n}{n^4}$.
\end{proposition}

\begin{proof}
 Recall  that   $\alpha$ was defined  in (\ref{defalpha}). 
Note that \begin{eqnarray*}  s_n & \geq &   \alpha \sum_{i=3}^n \frac{9^i}{i^4} \cdot {n\choose n-i} \\
& \geq & \frac{\alpha}{n^4}  \sum_{i=3}^n 9^i {n\choose n-i} \\
& \geq & \frac{c_1}{n^4}  \sum_{i=0}^n 9^i {n\choose n-i} \\
& = & \frac{c_1}{n^4} \cdot 10^n,
\end{eqnarray*}
where we used the binomial theorem in the last step.
\end{proof}

{\bf Remark.} If sequences  $\{a_n\}_n$ and $\{b_n\}_n$ of positive real numbers are related by the equality $b_n=\sum_{k=0}^n {n\choose k}a_k$, then the sequence
$\{b_n\}_n$ is sometimes called the {\em binomial transform} of the sequence $\{a_n\}_n$. This is a well-studied transform that is the subject of the book \cite{b18}. 
In particular, if $A(z)$ and $B(z)$ are the respective generating functions of the two sequences, then 
\[B(z)=\frac{1}{1-z} A\left(\frac{z}{1-z}\right).\] This equality could be used to analyze the singularities of $B(z)$ and therefore, to obtain the asymptotics of its coefficients.

Returning to the task at hand, the proof of the main theorem of this section is now immediate.
\begin{theorem} \label{nonalg-theorem}
The generating function $A_{A_{5,5}}(z)$ is not algebraic.
\end{theorem}

\begin{proof} Lemma \ref{firstchain} shows that $|\av_n(A_{5,5})|$ is between two constant multiples of $s_n$, while Propositions \ref{first} and \ref{second} prove that
$s_n$ is between two constant multiples of $10^n/n^4$. Therefore, there exist absolute constants $c>0$ and $C>0$ so that 
\begin{equation} \label{squeeze-chain} c\cdot \frac{10^n}{n^4} \leq |\av_n(A_{5,5})| \leq  C\cdot \frac{10^n}{n^4} \end{equation} for all $n$.

Therefore, by Lemma \ref{bostan},  $A_{A_{5,5}}(z)$ is not algebraic. \end{proof}

\subsection{The case of length $k$}
For general $k\geq 3$, the methods that we used in the last section yield the following.
\begin{theorem} \label{general}
Let $k\geq 3$. Then there are  absolute constants $c_k$ and $C_k$ so that the chain of inequalities
\[c_k \cdot \frac{((k-2)^2+1)^n}{n^{(k^2-4k+3)/2} } \leq \mid \av_n(A_{k,k}) \mid  \leq C_k \cdot \frac{((k-2)^2+1)^n}{n^{(k^2-4k+3)/2} } \]
\nopagebreak
holds.
\end{theorem}

\begin{proof} Analogous to the proof of (\ref{squeeze-chain}) in the proof of Theorem \ref{nonalg-theorem}. In that theorem, we had $k=5$, which yielded $(k-2)^2+1=10$. Note that Lemma \ref{firstchain} and Propositions \ref{first} and \ref{second} all \hspace{-0.2em}generalize \hspace{-0.2em}for \hspace{-0.2em}larger $k$.
\qedhere
\end{proof}

When $k=3$, then $A_{3,3}=\{132,231\}$, and it is well-known (see, for instance, Exercise 14.2 in \cite{b23}) that $|\av_n(A_{3,3})|=2^{n-1}$, in accordance with Theorem \ref{general}.
When $k=4$, then $A_{4,4}=\{1243,1342,2341\}$. It follows from Theorem 3.1. in \cite{m16} that
\[A_{A_{4,4}}(z)=\frac{1+z-\sqrt{1-6z+5z^2}}{2(2z-z^2)} .\] It follows from this formula that 
\[|\av_n(A_{4,4})|\simeq C \cdot \frac{5^n}{n^{3/2}},\] for some absolute constant $C$, 
again   in accordance with Theorem \ref{general}. See Sequence A033321 in \cite{sloane} for the many occurrences
of the sequence $|\av_n(A_{4,4})|$.

For larger $k$, we have the following generalization of Theorem  \ref{nonalg-theorem}.
\begin{theorem} Let $k$ be an odd integer so that $k>3$ holds.
Then the generating function $A_{A_{k,k}}(z)$ is not algebraic.
\end{theorem}

\begin{proof} If $k$ is an odd integer, then $(k^2-4k+3)/2$ is an integer. If, in addition, the inequality  $k>3$ holds,   then $(k^2-4k+3)/2$ is a {\em positive} integer. 
This implies that  we can apply  Lemma \ref{bostan}  using  the upper and lower bounds that we  obtain from  Theorem \ref{general},  completing \hspace{-0.2em}the proof of \hspace{-0.2em}the present theorem.
\qedhere
\end{proof} 

\section{Classes Wilf-equivalent to $\av(A_{k,k})$.}

Let $S$ and $T$ be two sets of patterns.
We say that the permutation classes $\av(S)$ and $\av(T)$ are {\em Wilf-equivalent} if for all $n$, the equality $|\av_n(S)|=|\av_n(T)|$ holds. 
Let $B_{5,5}=\{21354,21453, 31452,32451\}$, and define $B_{k,k}$ in an analogous way. That is, $B_{k,k}$ is the set of $k-1$ patterns of length
$k$ whose first $k-1$ entries form a $213\cdots (k-1)$ pattern, and whose last entry is less than $k$. 

\begin{theorem}  \label{thebees}
For all $k\geq 4$, the classes $\av(A_{k,k})$ and $\av(B_{k,k})$ are Wilf-equivalent. In particular, for all odd $k\geq 5$, the generating function
\[A_{B_{k,k}}(z)=\sum_{n\geq 0}| \av_n(B_{k,k})| z^n \] is not algebraic.
\end{theorem}

For $k=4$, the claim of Theorem \ref{thebees} is proved in \cite{bs17}. Our argument can be viewed as a generalization of that proof, but it is self-contained so that the 
reader does not have to understand the more general terminology and structural analysis presented in that paper.
The key element in our proof is the following lemma. It was stated and proved in a different form by Julian West in \cite{w90}.
The precise form in which we state and prove it is implicit in \cite{w90}, but for our purposes, the explicit form below is necessary.
 
Note that if $p=p_1p_2\cdots p_n$ is a permutation, then we say that $p_i$ is a {\em right-to-left maximum} of $p$ if there is no $j>i$ so 
that $p_j>p_i$. In general, if the longest increasing subsequence of $p$ that starts at $p_i$ is of length $m$, then we say that $p_i$ is of {\em co-rank} $m$.
So right-to-left maxima are precisely the entries of co-rank 1. 

\begin{lemma} \label{fixing}  Let $\ell\geq 3$. Then there exists a bijection  $g_{n,\ell}:=\av_n(123\cdots \ell) \rightarrow \av_n(213\cdots \ell)$ so that
for all $p\in  \av_n(123\cdots \ell) $, and for all $m\leq \ell-2$, the permutations $p$ and $g_{n,\ell}(p)$ have the same set of entries of co-rank $m$, and those entries are in the same positions.
In other words, $g_{n,\ell}$ leaves the entries  of $p$ that are of co-rank $m$ fixed.
\end{lemma}

Note that Lemma \ref{fixing} proves in particular that the two permutation classes above are Wilf-equivalent, but we will need the much stronger claim of the Lemma. 
Also note that for $k=3$, Lemma \ref{fixing} and its proof reduce to those of the classic Simion-Schmidt bijection \cite{ss85}.

\begin{proof} (of Lemma \ref{fixing}).
Let $p\in \av_n(123\cdots \ell)$. Leave the entries of $p$ that are of co-rank $\ell-2$ or less unchanged. Fill the remaining slots with the remaining entries, going from right to left, so that in every step, we place the largest entry that can be placed at the given position without getting co-rank $\ell-2$ or less. This results in a permutation $g_{n,\ell}(p)\in    \av_n(213\cdots \ell)$.
The map $g_{n,\ell}(p)$ is a bijection, because it has an inverse. Indeed, if $w\in  \av_n(213\cdots \ell)$, we get the unique preimage of $w$ by leaving its entries of co-rank $\ell-2$ or less fixed and writing the remaining entries into the \hspace{-0.2em}remaining slots in \hspace{-0.2em}decreasing order. 
\end{proof}

\begin{proof} (of Theorem \ref{thebees}).  We construct a bijection $h_n:\av_n(A_{k,k}) \rightarrow \av_n(B_{k,k})$. 
Let $p=p_1p_2\cdots p_n$, and let $i$ be the location of the rightmost descent of $p$, that is, the largest index so that $p_i>p_{i+1}$. Note that this means that the string
$p_1p_2\cdots p_i$ avoids the increasing pattern $12\cdots (k-1)$. We define $h_n(p)$ as the
concatenation of $g_{i,k-1}(p_1p_2\cdots p_i)$ and the increasing sequence $p_{i+1}\cdots p_n$. It is clear that $h_n(p)\in \av_n(B_{k,k})$, since if $h_n(p)$ contains a copy of a 
$21\cdots (k-1)$-pattern, that copy must end strictly on the right of position $i$, and $h_n(p)$ increases in all those positions.  

In order to prove that $h_n:\av_n(A_{k,k}) \rightarrow \av_n(B_{k,k})$ is indeed a bijection, first note that the rightmost descent of $h_n(p)$ is also in position $i$. As we know that 
$h_n$ fixes the increasing subsequence $p_{i+1}\cdots p_n$, it suffices to prove that if $x$ is the entry  in position $i$ of $h_n(p)$, then $x>p_{i+1}$. Note that $x$ is the rightmost entry
of $g_{i,k-1}(p_1p_2\cdots p_i)$. Recall that $g_i$ fixes all entries of $p_1p_2\cdots p_i$, except those that have co-rank $k-2$ in the string $p_1p_2\cdots p_i$. On the other hand, 
$p_i$ is the rightmost entry in $p_{i+1}\cdots p_n$, so its co-rank there is $1<k-2$.  So  we are done, since
$x=p_i>p_{i+1}$. 

We can now prove that the function $h_n:\av_n(A_{k,k}) \rightarrow \av_n(B_{k,k})$ is indeed a bijection by showing that it has an inverse. Let $w=w_1w_2\cdots w_n\in  \av_n(B_{k,k})$.
Let us find  the largest index $i$ so that the inequality  $w_i>w_{i+1}$ holds. We can then obtain  the unique permutation $p$ that satisfies  the equality $h_n(p)=w$  by taking the permutation $g_i^{-1}(w_1w_2\cdots w_i)$ and postpending it by string
 $p_{i+1}p_{i+2}\cdots p_n$.  \hspace{-31pt}
\end{proof}

\section{Further directions}
There are other several other patterns that are Wilf-equivalent with the monotone pattern $12\cdots k$. For our purposes, the following result is particularly relevant.
\begin{theorem} \label{west-backelin} Let $2\leq m\leq k-1$. Then for all $n$,  the equality
\[|\av_n(123\cdots k)| =| \av_n (m(m-1)\cdots 21 (m+1)(m+2)\cdots k)| \]
holds. 
\end{theorem}

\begin{proof}
For two patterns $q$ and $q'$ of length $\ell$ and $\ell'$ respectively,  let $q\oplus q'$ denote the pattern of length $\ell +\ell'$ whose first $\ell$ entries form a copy of the pattern $q$, 
whose last $\ell'$ entries form a copy of the pattern $q'$, and in which the set of the first $\ell$ entries is the set $\{1,2,\cdots ,\ell\}$, and so the set of the last $\ell'$ entries is
necessarily $\{\ell+1,\ell+2,\cdots ,\ell+\ell'\}$.  It then follows from results in \cite{bw00}, \cite{bwx07}, and \cite{k06} that
 for any pattern $q$, the equality $|\av_n(i_m \oplus q)| =|\av_n(d_m \oplus q)|$ holds, where $i_m$ is the increasing pattern of length $m$ and $d_m$ is the decreasing pattern of length $m$.
Then Theorem \ref{west-backelin} is the special case when $q=i_{k-m}$. 
The interested reader can check the details and the necessary notion of {\em shape-Wilf-equivalence} on pages 762 -763 of \cite{v15}. 
\end{proof}

In other words, we can reverse the subsequence of the first $m$ entries of the monotone pattern and get a pattern that is Wilf-equivalent to the original monotone pattern.
Unfortunately, Theorem \ref{west-backelin} by itself is not enough for our purposes. We would need an affirmative answer to the following question.

\begin{question} \label{bijective} 
 Does there exist a bijection \[G_{n,m}:\av_n(123\cdots k) \rightarrow  \av_n (m(m-1)\cdots 21 (m+1)(m+2)\cdots k) \] that fixes all right-to-left maxima?
\end{question}

If such a bijection exists, then the proof of Theorem \ref{thebees} can be extended as follows. Let $k>3$ be an odd integer, and let $1<m<k-1$. Let $A_{k,k,m}$ be the
set of $k-1$ patterns of length $k$ that start with an $m(m-1)\cdots 21 (m+1)(m+2)\cdots (k-1)$-pattern, and end in an entry less than $k$. So for instance, 
$A_{5,5,3}=\{32154, 42153, 43152, 43251\}$. Then the generating function $A_{A_{k,k,m}}(z)$ is not algebraic. 

It is possible that the answer to Question \ref{bijective} is positive, but difficult. Indeed, such an answer would yield a result that is stronger than Theorem \ref{west-backelin}, and that
theorem is quite difficult to prove on its own.

\acknowledgements
\label{sec:ack}

The author is indebted to Robert Brignall for pointing out useful references. The author is also indebted to his three referees whose careful reading improved the presentation of the results in this paper.

\bibliographystyle{abbrvnat}
\bibliography{finallongtail}

\begin{thebibliography}{19}
\providecommand{\natexlab}[1]{#1}
\providecommand{\url}[1]{\texttt{#1}}
\expandafter\ifx\csname urlstyle\endcsname\relax
  \providecommand{\doi}[1]{doi: #1}\else
  \providecommand{\doi}{doi: \begingroup \urlstyle{rm}\Url}\fi

\bibitem[Babson and West(2000)]{bw00}
E.~Babson and J.~West.
\newblock The permutations $123p_4\cdots p_m$ and $321p_4\cdots p_m$ are
  wilf-equivalent.
\newblock \emph{Graphs Combin.}, 16\penalty0 (4):\penalty0 373--380, 2000.

\bibitem[Backelin et~al.(2007)Backelin, West, and Xin]{bwx07}
J.~Backelin, J.~West, and G.~Xin.
\newblock Wilf-equivalence for singleton classes.
\newblock \emph{Adv. Appl. Math.}, 38\penalty0 (2):\penalty0 133--148, 2007.

\bibitem[B\'ona(2020)]{b20}
M.~B\'ona.
\newblock Supercritical sequences, and the nonrationality of most principal
  permutation classes.
\newblock \emph{European J. Combin.}, 83\penalty0 (103020):\penalty0 8, 2020.

\bibitem[B\'ona(2023)]{b23}
M.~B\'ona.
\newblock \emph{A Walk Through Combinatorics}.
\newblock World Scientific, 5th edition, 2023.

\bibitem[B\'ona and Burstein(2022)]{bb22}
M.~B\'ona and A.~Burstein.
\newblock Permutations with exactly one copy of a monotone pattern of length
  $k$, and a generalization.
\newblock \emph{Ann. Comb.}, 26\penalty0 (2):\penalty0 393--404, 2022.

\bibitem[B\'ona and Pantone(2023)]{bp23}
M.~B\'ona and J.~Pantone.
\newblock Permutations avoiding sets of patterns with long monotone
  subsequences.
\newblock \emph{J. Symbolic Comput.}, 116:\penalty0 130--138, 2023.

\bibitem[Bostan(2021)]{bos21}
A.~Bostan.
\newblock personal communication, 2021.

\bibitem[Boyadzhiev(2018)]{b18}
K.~N. Boyadzhiev.
\newblock \emph{Notes on the Binomial transform}.
\newblock World Scientific, 2018.

\bibitem[Brignall and Slia\v{c}an(2017)]{bs17}
R.~Brignall and J.~Slia\v{c}an.
\newblock Juxtaposing catalan permutation classes with monotone ones.
\newblock \emph{Electron. J. Combin.}, 24, 2017.

\bibitem[Garrabrant and Pak(2015)]{gp}
S.~Garrabrant and I.~Pak.
\newblock Pattern avoidance is not p-recursive.
\newblock Preprint, 2015.
\newblock URL \url{https://arxiv.org/pdf/1505.06508.pdf}.

\bibitem[Jungen(1931)]{j31}
R.~Jungen.
\newblock Sur les s\'eries de taylor n'ayant que des singularit\'es
  alg\'ebrico-logarithmiques sur leur cercle de convergence.
\newblock \emph{Commentarii Mathematici Helvetici}, 3:\penalty0 266--306, 1931.

\bibitem[Krattenthaler(2006)]{k06}
C.~Krattenthaler.
\newblock Growth diagrams, and increasing and decreasing chains in fillings of
  ferrers shapes adv.
\newblock \emph{Adv. Appl. Math.}, 37\penalty0 (3):\penalty0 404--431, 2006.

\bibitem[Miner(2016)]{m16}
S.~Miner.
\newblock Enumeration of several two-by-four classes.
\newblock Preprint, 2016.
\newblock URL \url{https://arxiv.org/pdf/1610.01908.pdf}.

\bibitem[Regev(1981)]{r81}
A.~Regev.
\newblock Asymptotic values for degrees associated with strips of young
  diagrams.
\newblock \emph{Advances in Mathematics}, 41:\penalty0 115--136, 1981.

\bibitem[Simion and Schmidt(1985)]{ss85}
R.~Simion and F.~Schmidt.
\newblock Restricted permutations.
\newblock \emph{European J. Combin.}, 6\penalty0 (4):\penalty0 383--406, 1985.

\bibitem[Sloane(2024)]{sloane}
N.~J.~A. Sloane.
\newblock Online encyclopedia of integer sequences, 2024.
\newblock URL \url{http://oeis.org}.
\newblock Accessed on June 19, 2024.

\bibitem[Stanley(2023)]{s23}
R.~Stanley.
\newblock \emph{Enumerative Combinatorics, Volume II}.
\newblock Cambridge University Press, 2nd edition, 2023.

\bibitem[Vatter(2015)]{v15}
V.~Vatter.
\newblock Permutation classes.
\newblock In M.~B\'ona, editor, \emph{Handbook of Enumerative Combinatorics}.
  CRC Press, 2015.

\bibitem[West(1990)]{w90}
J.~West.
\newblock \emph{Permutations with forbidden subsequences and stack-sortable
  permutations}.
\newblock PhD thesis, Massachusetts Institute of Technology, 1990.

\end{thebibliography}
\label{sec:biblio}

\end{document}